\documentclass[fleqn]{amsart}

\usepackage[boxed,norelsize]{algorithm2e}

\newtheorem{theorem}{Theorem}
\newtheorem{lemma}[theorem]{Lemma}
\newtheorem{corollary}[theorem]{Corollary}
\theoremstyle{definition}
\newtheorem*{definition*}{Definition}
\newtheorem{proposition}[theorem]{Proposition}
\newtheorem*{remark}{Remark}

\newcommand{\N}{{\mathbb{N}}}
\newcommand{\floor}[1]{\lfloor #1 \rfloor }

\begin{document}

\everymath{\displaystyle}

\title[Levin's construction of absolutely normal numbers]{M. Levin's construction of absolutely normal numbers  with very low discrepancy}

\author{Nicol\'as Alvarez}
\address{Departamento de Ciencias e Ingenier\'ia de la Computaci\'on, Universidad Nacional del Sur, Argentina}
\email{naa@cs.uns.edu.ar}
\thanks{Supported by a doctoral fellowship from CONICET, Argentina.}

\author{Ver{\'o}nica\  Becher}
\address{Departmento de  Computaci\'on, Facultad de Ciencias Exactas y Naturales, Universidad de Buenos Aires \& CONICET, Argentina}
\email{vbecher@dc.uba.ar}
\thanks{Supported by Agencia Nacional de Promoci\'on Cient\'ifica y Tecnol\'ogica and CONICET, Argentina.}



\date{\today}

\maketitle

\begin{abstract}
Among the currently known constructions of absolutely normal numbers,  
the one given by Mordechay~Levin  in 1979 achieves the lowest discrepancy bound.  
In this work we analyze this construction in terms of computability and computational complexity. 
We show that, under basic assumptions, it yields a   computable real number.
The construction  does not give   the digits of  the  fractional expansion explicitly, 
but it gives  a  sequence of increasing approximations whose limit is the announced absolutely normal number.
The $n$-th approximation has an error  less than $2^{2^{-n}}$.
To obtain the $n$-th approximation  the construction requires, in the worst case,  
a number of mathematical operations that  is double  exponential in~$n$. 
We consider  variants on the construction that reduce the computational complexity 
at the expense of an increment in discrepancy.
\end{abstract}

\section{Introduction}
Normal numbers were introduced by Borel in~1909~\cite{borel1909}.
A real number $\alpha$  is normal to an integer base $\lambda$ greater than or equal to~$2$
 if  its fractional expansion in base $\lambda$ given by 
\[
\alpha -\lfloor \alpha \rfloor=\sum_{x\geq 1} \frac{d_x}{ \lambda^{x}}\qquad \text{ where  each   $d_x$ is in $\{0,1,...,\lambda -1\}$,}
\]
is such that, for each positive integer $k$,  each fixed  block of digits of length $k$  appears in $(d_x)_{x\geq 1}$ with asymptotic frequency~$\lambda^{-k}$.
Borel calls a number  absolutely normal if it  is normal to every integer base greater than or equal to~$2$.

Let $(\xi_{x})_{x\geq 0}$ be an arbitrary sequence of real numbers in the unit interval. 
The quantity
\[
D(P,(\xi_{x})_{x\geq 0})= \sup_{\gamma \in (0,1]} \left| \frac{\#\{ x : 0\leq x<P  \text{ and } \xi_x< \gamma\}}{P} - \gamma \ \right| 
\]
is  the discrepancy of $(\xi_x)_{x=0}^{P-1}$.
The sequence $(\xi_{x})_{x\geq 0}$ is  uniformly distributed  in  the unit interval 
 if $D(P, (\xi_{x})_{x\geq 0})$ goes to $0$ when $P$ goes to infinity.
By a theorem of D.~Wall 
\cite[Theorem 4.14]{bugeaud2012},  a real number $\alpha$ is normal to base 
$\lambda$  if, and only if, the sequence $\{\alpha \lambda^x\}_{x\geq 0}$, 
where $\{\xi\} = \xi- [\xi] $ is the fractional part of $\xi$,
 is uniformly distributed  in the unit interval. 

We  use the customary notation for asymptotic growth of functions,
and we say $f(n)$ is in $O(g(n))$  if  $\exists k>0 \ \exists n_0 \ \forall n>n_0 $, $ |f(n)| \leq k |g(n)|$.

Borel \cite{borel1909} proved that 
almost every real number (in the sense of Lebesgue measure) is normal to every integer base
and  Gal and Gal \cite{galgal} showed that, indeed, 
for almost every real number $\alpha$  and for every integer base $\lambda$
 the discrepancy  $D(P, \{\alpha \lambda^x\}_{x\geq 0}) $  is in $O\Big(\sqrt{\frac{\log\log P}{P}}\Big)$. 
For a thorough presentation of normal numbers and the theory of 
uniform distribution see the books~\cite{bugeaud2012,kuipers,drmotatichy1997}.

In 1979 Mordechay Levin   \cite{levin} 
considered  the notion of normality for real numbers  with respect to  bases that are real numbers greater than~$1$,
and  he gave   an explicit construction of  a number that is normal   to arbitrary many real bases, with controlled discrepancy of normality. 
More precisely, given a sequence   $(\lambda_j)_{j\geq 1}$ of real numbers greater than~$1$, 
 a monotone increasing  sequence 
 $(t_j)_{j\geq 1}$  of positive integers and a non negative real number $a$,
 Levin constructs a real number $\alpha$ greater than $a$ that is normal to each of the bases $\lambda_j$, for $j=1, 2, \ldots$
such that  $D(P, \{\alpha \lambda^x\}_{x\geq 0}) $ is in 
$O\Big(\frac{(\log P)^2}{\sqrt{P}}\omega(P)\Big)$,
where $\omega(P)$ is a non-decreasing function  
that determines from $(\lambda_j)_{j\geq 1}$ and $(t_j)_{j\geq 1}$ 
 the collection of bases  considered at position~$P$,
and the constant in the order symbol depends on $\lambda_j$.
Since $(\lambda_j)_{j\geq 1}$ and $(t_j)_{j\geq}$ can be such that   $\omega(P)$  grows arbitrarily slow, 
so $D(P, \{\alpha \lambda^x\}_{x\geq 0}) $  can end up being in $O\Big(\frac{(\log P)^2}{\sqrt{P}}\Big)$.
By considering normality  with respect to arbitrary  sequences $(\lambda_j)_{j\geq 1}$ of real numbers greater than $1$, 
Levin extends Borel's  notion of absolute normality.
With~$\lambda_j=j+1$ for  $j=1, 2,\ldots$,  he obtains a number $\alpha$ that is 
absolutely normal in Borel's~sense.

The interest  in  this construction by Levin is that, among the currently known 
methods to  construct absolutely normal numbers, 
 it  achieves the lowest discrepancy bound.  
In this work we give a plain presentation of  this construction and 
we show that, under basic assumptions,  the construction is computable and we establish its computational complexity.
 
Regarding  discrepancy and  computational complexity, known constructions of 
computable absolutely normal numbers can be classified as follows:


$\bullet$ \ Constructions that run in double exponential time, which means that 
to produce the $P$-digit of the expansion of the constructed number $\alpha$ in a given base 
they perform a number of operations that is double exponential in $P$. 
One example is Alan Turing's algorithm  \cite{turing,BecherFigueiraPicchi2007} 
for which $D(P, \{\alpha \lambda^x\}_{x\geq 0}) $ is  in $O\Big(\frac{1}{\sqrt[16]{P}}\Big)$.
Another is  the computable reformulation of Sierpi\'nski's construction~\cite{BecherFigueira2002}
for which $D(P, \{\alpha \lambda^x\}_{x\geq 0}) $ is in ${O\Big(\frac{1}{\sqrt[9]{P}}\Big)}$.

$\bullet$ \ Constructions that run in exponential time, 
as  Wolfgang Schmidt's  algorithm~\cite{schmidt} 
for which $D(P, \{\alpha \lambda^x\}_{x\geq 0}) $ is in
$O\Big(\frac{(\log P)^4 }{e^{({\log P)}^{1/4}}}  \Big)$.
Our variants of Schmidt's  algorithm \cite{BecherBugeaudSlaman2013,BecherSlaman2014} 
also require exponential time. 
These algorithms produce numbers that are normal to all the bases in a given arbitrary  set, 
while they are not (simply) normal to any of the multiplicatively independent 
bases in the complement.
Besides, our algorithm  \cite{liouville}  for computing an absolutely normal Liouville number $\alpha$
has at least exponential complexity and we have not estimated the 
discrepancy of the sequence $\{\lambda^j \alpha\}_{j= 0}^{P-1}$, for positive~$P$.

$\bullet$ \ Constructions that run in polynomial time, 
as our algorithm \cite{poly} that requires just above quadratic time
to compute an absolutely normal number~$\alpha$. 
Speed of computation is obtained by sacrificing discrepancy. 
The algorithm deals explicitly  with  the discrepancy at the intermediate steps of the construction
but we have  not estimated the discrepancy of the sequence~$\{\lambda^x \alpha\}_{x\geq  0}$.

There are constructions of numbers ensuring normality to just one base which achieve
much lower discrepancy bounds than those  for absolute normality.
The one with smallest discrepancy was given also by  Levin \cite{levin99}.
Using van der Corput type sequences.
Levin constructs a number  $\alpha$ normal to an integer base  $\lambda$, 
such that the discrepancy  $D(P, \{\alpha \lambda^x\}_{x\geq 0}) $ is in  $O\Big(\frac{(\log P)^2}{P} \Big)$.  
This discrepancy bound is surprisingly small, considering that 
for any sequence $(\xi_x)_{x\geq 0}$ of reals in the unit interval,
\[
\limsup_{P\to\infty}  \frac{P}{\log P}  D(P,(\xi_{x})_{x\geq 0}) 
\]
is greater than $0$ (this result was proved by W. Schmidt in 1972, see \cite{bugeaud2012}).  
The computational complexity of this construction  has not been studied yet.
Recently, Madritsch and Tichy \cite{madritschtichy2015} found conditions for van der Corput sets and 
 suggest to use them for the construction of absolutely normal~numbers.

The construction insuring normality to one base
that has smallest   computational complexity  
coincides with the historically first construction of a number that is normal to base $10$, 
and it is due to  Champernowne in 1933 \cite{champernowne}. 
Champernowne's constant is computable with  logarithmic complexity,
which means that the $P$-th digit in the  expansion  can be obtained  
independently of all the previous digits by performing 
$O(\log P)$ elementary operations. It is also possible  to compute the first 
$P$ digits of Champerowne's constant in $O(P)$ operations.
The  discrepancy $D(P, \{Champernowne's\ constant\ \lambda^x\}_{x\geq 0}) $ is in
 $O\Big(\frac{1}{\log P}\Big)$ and for every $P$ it  has been proved to be greater than 
or equal to  $\frac{K}{\log P}$, for positive $K$ \cite{champernowne,levin99, schiffer}.

\section{Levin's construction}

In this section we  give a comprehensible presentation of  Levin's  construction~\cite{levin}. 
We reorganized the original material but we kept the notation.

\begin{definition*}
Let $\lambda$ be a real number greater than~$1$ and let 
 $(\lambda_j)_{j=1}^\infty$ a sequence of real numbers, each  greater than~$1$.
A number $\alpha$ is normal to base $\lambda$
if the sequence $\{\alpha\lambda^x\}_{x\geq 0}$ is uniformly distributed in the unit interval, 
and absolutely normal  to bases  $\lambda_j$ for all positive $j$, if,  
 $\alpha$ is normal to base~$\lambda_j$ for each positive $j$.  
\end{definition*}

\begin{theorem}[Levin \cite{levin}]\label{thm:levin}
Let $(\lambda_j)_{j\geq 1}$ be sequence of real numbers  greater than~$1$,  
  let $(t_j)_{j\geq 1}$ be a sequence of integers monotonically increasing at any speed
and let $a$ be a non-negative real number.
There is a real number $\alpha$  constructed  from $a$ and the sequences $(\lambda_j)_{j\geq 1}$ and $(t_j)_{j\geq 1}$  which 
 is normal  to  base~$\lambda_j$ for each positive integer~$j$ 
and  such that for any positive integer~$P$,
\[
    D(P, \{\alpha \lambda_j^x\}_{x\geq 0})  \text{ is in }  O\left(\frac{(\log P)^2 \omega(P)}{\sqrt{P}}\right)
\]
where
$\omega(P)\! = 1 \text{  if } P \in [1,\ell_2), \text{and }
  \omega(P)\!= k \text{  if  } P \in [\ell_k, \ell_{k+1})$,
\text{with } 
\\$\ell_k\!=\max(t_k,\max_{1 \leq v\leq k} 2\lceil |\log_2\log_2 \lambda_v|\rceil + 5)$
and the constant in the order symbol depends on~$\lambda_j$.
\end{theorem}

\SetKwFor{Repeat}{repeat}{}{end}
\begin{algorithm}

\BlankLine
\SetAlgorithmName{Construction}{}
\BlankLine

\TitleOfAlgo{\bf Levin's construction of absolutely normal numbers}

\BlankLine
\BlankLine
  
 \SetKwInOut{Input}{input}\SetKwInOut{Output}{output}
\DontPrintSemicolon
\Input{a sequence $(\lambda_j)_{j\geq 1}$ of reals greater than $1$,\;  
\hspace*{1.3cm}   an increasing sequence $(t_j)_{j\geq 1}$ of integers\; 
\hspace*{1.3cm} a  non-negative real  $a$.}

\BlankLine

\Output{a sequence of rationals $(\alpha_r)_{r\geq 1}$ such  that 
$\lim_{r\to \infty} \alpha_r = \alpha$ and \;
for each $\lambda_j$, the discrepancy of $\{\alpha \lambda_j^x\}_{x=0}^{P}$ is
 in $O\Big(\frac{(\log P)^2 \omega(P)}{\sqrt{P} }  \Big)$.}

\BlankLine
\BlankLine

Define the function  $\displaystyle \ell_k = \max(t_k, \max_{1 \leq v \leq k} 2 \lceil |\log_2\log_2 \lambda_v| \rceil + 5)$\;

 $r = \ell_1$ \;

 $\alpha_r= a$  \;

\BlankLine

\Repeat{forever}{
\Indp
	$n_r= 2^r-2 $\;
	$q_r =  2^{2^r+r+1}$\;
        if  $r $ in $[1,\ell_2)) $  then  $\omega(r) = 1$ \; 
        else $\omega(r)$ = the  unique  $k $ such that  $r $ in $ [\ell_k, \ell_{k+1})$ \;

	\For{$j=1 $ \KwTo $\omega(r)$}{
		$\tau_{r,j}=n_{r+1,j}-n_{r,j}$\;
		$A_{r,j}=\lfloor\sqrt{ \tau_{r,j}}\rfloor$\;
	}
	\mbox{Find the least integer $a_r$ in $[0, q_r)$ such that for each $j$ in $[1,\omega(r)]$}\;
	\centerline{$\displaystyle{D_{r,j}(a_r)< 2\Big( \frac{\lambda_j}{\lambda_j-1}\Big)^{3/2} \sqrt{\tau_{r,j}}\big(3+\ln \tau_{r,j}\big)^2} $}\
	where \;
	\qquad $D_{r,j}(c)=\sideset{}{'}\sum_{m_1,m_2=-A_{r,j}}^{A_{r,j}} \frac{|S_{r,j}(m_1,m_2,c)|}{\overline{m_1}\ \overline{m_2}}$,\;
       	\qquad $S_{r,j}(m_1,m_2,c)=\sum_{x=0}^{{\tau_{r,j}}-1} 
	 e\left(2 \pi i \Big(m_1\Big(\alpha_r+\frac{c}{2^{n_r} q_r }\Big) \lambda_{j}^{n_{r,j}+x} + \frac{m_2x}{\tau_{r,j}}\Big)\right)$,\;
       \qquad $\sideset{}{'}\sum$ denotes the sum without  the term with $m_1=m_2=0$, \;
       \qquad $\overline{m}=\max(1,|m|)$. \;
       
\BlankLine
\BlankLine
 
$\displaystyle{\alpha_{r+1}=\alpha_r + \frac{a_r}{2^{n_{r}} {q_r}}}$\;

print $\alpha_{r+1}$\;

$r= r+1$\;
\BlankLine
\BlankLine

}
\end{algorithm}

The number $\alpha$ proposed by Levin is defined as
\[
\alpha= a + \sum_{r=\ell_1}^{\infty}\frac{a_r}{2^{n_r} q_r},
\]
\begin{itemize}
\item[] $n_r=2^r-2$, and 
\item[] $q_r= 2^{2^r+r+1}$.
\end{itemize}
Fix $(\lambda_j)_{j\geq 1}$ an arbitrary  sequence of real numbers greater than~$1$,
fix $(t_j)_{j\geq 1}$  a sequence of integers monotonically increasing at any speed and fix a non-negative real $a$.
Along the article we refer freely to the values $\ell_r$,  $n_r, q_r$, $a_r$ and $\omega(r)$ for any positive $r$
as well as  to the real  $\alpha$.

We need some further notation.
For each pair of positive integers  $r,j$ we let
\begin{itemize}
\item[]  $n_{r,j}=\floor{n_r \log_{\lambda_j }2}$,
\item[] $\tau_{r,j}=n_{r+1,j}- n_{r,j}$, and
\item[] $A_{r,j}=\floor{\sqrt{\tau_{r,j}}}$.
\end{itemize}

\begin{lemma}\label{lemma:easy}
For every positive $j$ and for every $r\geq \ell_j-1$,
\begin{itemize}
\item[] $2^{r-1}\log_{\lambda_j}2 \leq \tau_{r,j}\leq 2^{r+1}\log_{\lambda_j} 2$,
\item[] $ \tau_{r,j} \geq \max(7, \tau_{r+1,j}/4)   $.
\end{itemize}
\end{lemma}
\begin{proof}
From the definitions we know that $\tau_{r,j}=2^{r}\log_{\lambda_j}2+\theta_{r,j}$, where $|\theta_{r,j}|\leq 1$,  while for $r\geq \ell_j-1$,
\[
8=  2^{\log_2\log_2 \lambda_j + 3}  \log_{\lambda_j} 2
\leq 2^r\log_{\lambda_j}2.
\]
The wanted inequalities follow.
\end{proof}

Fix $\alpha_{\mbox{$\ell_1$}}=a$ and for each positive integer $m$, let $ a_{m}$ in $[0,q_m)$.
For every $r\geq \ell_1$ define
\[
\alpha_{r+1}=\alpha_{\ell_1}+\sum_{m={\ell_1}}^r    \frac{a_m}{2^{n_{m}} {q_m}}.
\]
We write $e(x)$ to denote $e^x$.
For integers $c,m_1,m_2,r$ with $r\geq \ell_j$ we define the quantities,
\[
S_{r,j}(m_1,m_2,c)=\sum_{x=0}^{{\tau_{r,j}}-1} 
 e\left(2 \pi i \Big(m_1\Big(\alpha_r+\frac{c}{2^{n_r} q_r }\Big) \lambda_{j}^{n_{r,j}+x} + \frac{m_2x}{\tau_{r,j}}\Big)\right),
\]
\[
D_{r,j}(c)=\sideset{}{'}\sum_{m_1,m_2=-A_{r,j}}^{A_{r,j}} \frac{|S_{r,j}(m_1,m_2,c)|}{\overline{m_1}\ \overline{m_2}},
\]
where $\overline{m}=\max(1,|m|)$ and $\sideset{}{'}\sum$ denotes that the term with $m_1=m_2=0$ is absent from the sum.

\begin{remark}
In Levin's paper \cite{levin} the definition of $S_{r,j}(m_1,m_2,c)$ 
appears with  $\sideset{}{'}\sum$ while the definition of $D_{r,j}(c)$
appears with  $\sum$.
However, the use of $\sideset{}{'}\sum$  excludes the term $m_1=m2=0$,
which only makes sense in the definition of  $D_{r,j}(c)$.
\end{remark}

\begin{lemma}[Lemma 1 in \cite{levin}]\label{lemma:levin1}
Let integers $j, r,m_1,m_2$ such that $r\geq \ell_j$ and $0< \max(|m_1|, |m_2|)\leq A_{r,j}$.
Then, 
\[
\Big(\frac{ 1}{q_r} \sum_{c=0}^{{q_r}-1} | S_{r,j}(m_1,m_2,c)|^2 \Big)^{1/2} < 
2 \Big(\frac{\lambda_j}{\lambda_j-1}\Big)^{3/2}\sqrt{\tau_{r,j}}.
\]
\end{lemma}
\begin{proof}
Let
\[
    T_{r,j}(m_1,m_2)= \Big(\frac{ 1}{q_r} \sum_{c=0}^{{q_r}-1} | S_{r,j}(m_1,m_2,c)|^2 \Big)^{1/2}.
\]
\begin{remark}
Levin's original paper misses the third parameter $c$ of the function $S_{r,j}$.
\end{remark}
In accordance with the familiar inequality
\[
\frac{1}{N} \left|\sum_{x=0}^{N-1}e(2\pi i \theta x)\right|
\leq \min\left(1, \frac{1}{2N\langle\langle \theta\rangle\rangle}\right),
\]
where $\langle\langle \theta\rangle\rangle$ is the distance of $\theta$ from the nearest integer, we have 
\begin{align*}
&\hspace*{-1cm}
T^2_{r,j}(m_1,m_2)=
\\=&\sum_{x,y=0}^{\tau_{r,j}-1} \frac{1}{q_r} 
\sum_{c=0}^{q_r-1}
e\Big(2\pi i 
\Big(m_1 
\Big(\alpha_r+\frac{c}{2^{n_r}q_r}\Big)
\big(\lambda_j^{n_{r,j}+x}-\lambda_j^{n_{r,j}+y} \Big)+  
            \frac{m_2(x-y)}{\tau_{r,j}} \Big)\Big)
\\
<& \sum_{x,y=0}^{\tau_{r,j}-1} 
    \min\left(1, \frac{1}{2 q_r\langle\langle m_1 \frac{\lambda_j^{n_{r,j}+x}-\lambda_j^{n_{r,j}+y}}{2^{n_r} q_r}\rangle\rangle}\right).
\end{align*}
If $m_1$ equals $0$ then $m_2$ does not belong to $0(\mod \tau_{r,j})$,
$\tau_{r,j}\geq 7$,  $0 < |m_2|\leq A_{r,j}<\tau_{r,j}$, and 
\[
T_{r,j}(0,m_2)=0.
\]
Let $|m_1|>0$.
Let us show that the expression under $\langle\langle\rangle\rangle$ sign above has absolute value less than $1/2$.
Since $r\geq \ell_j$, by Lemma \ref{lemma:easy},
\begin{align*}
&\lambda_{j}^{n_{r+1,j}} \leq \lambda_{j}^{n_{r+1} log_{\lambda_{j}} 2} = 2^{n_{r+1}} = 2^{n_r} 2^{2^r},
\\
&\log_{\lambda_j}2= 2^{-\log_2 \log_2 \lambda_j }< 2^{\ell_j-3}< 2^{r-3},
\\
&A_{r,j}=\floor{\sqrt\tau_{r,j}}<\sqrt{2^{r+1}\log_{\lambda_j}2}< 2^{r-1}.
\end{align*}
Hence, 
\[
|m_1 (\lambda_j^{n_{r,j}+x}-\lambda_j^{n_{r,j}+y})|
<
2 A_{r,j} \lambda_j^{n_{r+1,j}} < 2^r 2^{n_r} 2^{2^r}= (1/2) 2^{n_{r}}q_r,
\]
and we can replace $\langle\langle\rangle\rangle$  by the absolute value sign:
\[
T^2_{r,j}(m_1,m_2)\leq 
\tau_{r,j}+2\sum_{\tau_{r,j} > x > y\geq 0} \frac{2^{n_r}}{2|m_1| \lambda_j^{n_{r,j}}(\lambda_j^x-\lambda_j^y)}.
\]
Using the definition of $n_{r,j}$, 
\[
\lambda_{j}^{n_{r,j}+1}\geq \lambda_j^{n_r\log_{\lambda_j}2}= 2^{n_r},
\]
whence,
\begin{align*}
T^2_{r,j}(m_1,m_2)&\leq \tau_{r,j} + \sum_{\tau_{r,j} > x > y\geq 0}  \frac{1}{\lambda_j^y \lambda_j^{x-y-1}(1-\lambda_j^{y-x})}
\\&< \tau_{r,j} +\sum_{y,z=0}^{\infty} \frac{1}{\lambda_j^y \lambda_j^z (1-\lambda_j^{-1})}
\\&= \tau_{r,j}+\Big(\frac{\lambda_j}{\lambda_j-1} \Big)^3
\\&<4\tau_{r,j}\Big(\frac{\lambda_j}{\lambda_j-1} \Big)^3.
\end{align*}
\end{proof}

\begin{lemma}[Lemma 2 in \cite{levin}]\label{lemma2}
Let $r\geq \ell_1$. There exists an integer $a_r$ in $[0,q_r)$ such that, 
given any positive integer $j$ and with the condition $r\geq \ell_j$, we have
\[
D_{r,j}(a_r)< 
2\Big( \frac{\lambda_j}{\lambda_j-1}\Big)^{3/2} 
\sqrt{\tau_{r,j}}\big(3+\ln \tau_{r,j}\big)^2 
\omega(r).
\]
\end{lemma}

\begin{proof}
Using the Cauchy-Bunyakovskii-Schwarz inequality  we obtain, 
\begin{align*}
\frac{1}{q_r}\sum_{c=0}^{q_r-1} D_{r,j}(c)
=& 
\sum_{m_1,m_2=-A_{r,j}}^{A_{r,j}}  \!\!\!\!\!\!\!\!^{\Large '} \ \ \ \ \  \frac{1}{\overline{m_1}\overline{m_2} q_r}\sum_{c=0}^{q_r-1} |S_{r,j}(m_1,m_2, c)|
\\
\leq& 
\sum_{m_1,m_2=-A_{r,j}}^{A_{r,j}}  \!\!\!\!\!\!\!\!^{\Large '} \ \ \ \ \  \frac{1}{\overline{m_1}\overline{m_2}}
\Big(\frac{ 1}{q_r} \sum_{c=0}^{{q_r}-1} | S_{r,j}(m_1,m_2)|^2 \Big)^{1/2}.
\end{align*}
Since the conditions of Lemma~\ref{lemma:levin1} are satisfied, we have
\begin{align*}
\frac{1}{q_r}\sum_{c=0}^{q_r-1} D_{r,j}(c) 
&< 2\Big(\frac{\lambda_j}{\lambda_j-1}\Big)^{3/2}\sqrt{\tau_{r,j}}(3+2\ln A_{r,j})^2
\\
&\leq 2 \Big(\frac{\lambda_j}{\lambda_j-1}\Big)^{3/2}\sqrt{\tau_{r,j}}(3+\ln \tau_{r,j})^2.
\end{align*}
Consequently, with $r\geq \ell_j$, the number of integers $c$ in $[0, q_r)$ such that 
\[
D_{r,j}(c)\geq  2 \omega(r) \Big(\frac{\lambda_j}{\lambda_j-1}\Big)^{3/2}\sqrt{\tau_{r,j}}(3+\ln \tau_{r,j})^2
\]
is less than $q_r/\omega(r)$.
By the definitions of $\omega(r)$ and $\ell_j$,
 conditions $r\geq \ell_j$  and $ \omega(r)\geq j$  are equivalent. 
In this case, the number of integers $c$ in $[0,q_r)$, such that the 
above inequality holds for at least one positive integer $j$, with the condition~$r\geq \ell_j$
(alternatively, $j\in [1,\omega(r)]$) is less than $\omega(r)\floor{q_r/\omega(r)}=q_r$.
Consequently, there exists an integer $c=a_r$ in $[0,q_r)$, such that the inequality in the statement 
of this lemma holds for all positive $j$ with the condition~$r\geq \ell_j$.
\end{proof}

For the proof of Theorem~\ref{thm:levin} Levin uses multidimensional discrepancy 
and applies  Koksma's inequality  \cite{koksma1950}.

Let $s$ be a positive integer, let $\gamma_v$, for $v=1, \ldots, s$, 
be real numbers in the unit interval, let $(\beta_{x,v})_{x\geq 0}$ for $v=1, \ldots, s$ 
be a set of real number sequences, and let 
$N_v(P)$ be the number of solutions for $x = 0, 1, \ldots , P-1$, of the system of inequalities
\begin{eqnarray*}
    \{\beta_{x,1}\} & < & \gamma_1 \\
    \{\beta_{x,2}\} & < & \gamma_2 \\
                  & \vdots  & \\
    \{\beta_{x,s}\} & < & \gamma_s. 
\end{eqnarray*}
The quantity 
\[
D(P,(\{\beta_{x,1}\}, \ldots ,\{\beta_{x,s}\})_{x\geq 0})=
\sup_{\gamma_1,\ldots,\gamma_s\in(0,1]^s} \left|\frac{N_v(P)}{P}-\gamma_1\cdot \cdot \gamma_s\right|
\]
 is called the discrepancy of the sequences $\{\beta_{x,1}\}, \ldots ,\{\beta_{x,s}\}$, for $x=0\ldots,P-1$.

\begin{lemma}[Koksma \cite{koksma1950}]\label{lemma:koksma}
Let $s$ be a positive integer, let $\gamma_v$, for $v=1, \ldots, s$, 
be real numbers in the unit interval, let $(\beta_{x,v})_{x\geq 0}$ for $v=1, \ldots, s$ be a set of real number sequences.
Let $P$ be a positive integer.  Then, for every  integer $n$,
\medskip\\
$
D(P,(\{\beta_{x,1}\}, \ldots ,\{\beta_{x,s}\})_{x\geq 0})
\leq 
{
30^s\left( \frac{1}{n}+\frac{1}{P} 
\sum_{m_1\ldots m_s=-n}^n 
\!\!\!\!\!\!\!\!^{\Large '} 
\frac{\Big|\sum_{x=0}^{P-1} e\Big(2\pi i \sum_{v=1}^s m_v \beta_{x,v}\Big)\Big|}{\overline{m_1}\ldots\overline{m_s}}\right).}
$
\end{lemma}

We can now present Levin's proof of Theorem~\ref{thm:levin} \cite{levin}.

\begin{remark}
In the next proof we write $n_{\ell_j,j}$ where Levin  wrote $n_{\ell_j}$.
\end{remark}

\begin{proof}[Proof of Theorem~\ref{thm:levin}]
For any three real numbers $\xi, \lambda,  \gamma$ and non-negative integers $Q$ and $P$, we denote by  
$N_{\xi, \lambda, \gamma}(Q,P)$ the number of solutions of the inequality
\[
\{\xi \lambda^x\}< \gamma, \quad \text{ for }x=Q, \ldots, Q+P-1.
\]
We write $N_{\xi, \lambda, \gamma}(P)$, to denote $N_{\xi, \lambda, \gamma}(0, P)$.

Fix any positive integer $j$ and any positive real  $\gamma$ in the unit interval.
Fix any positive integer $P$  and  define an integer $k$ from the condition $n_{k,j}\leq P < n_{k+1,j}$.
Then, 
\[
P=n_{k,j}+R_1, \text{ where }0\leq R_1< \tau_{k,j}.
\]
Observe that when $P$ is large enough, $k\geq \ell_j$.
Using the  definition of $N_{\alpha, \lambda_j,\gamma}$,
\[
N_{\alpha, \lambda_j,\gamma}(P)= N_{\alpha, \lambda_j,\gamma}(n_{\ell_j,j})+\sum_{r=\ell_j}^k N_{\alpha, \lambda_j,\gamma}(n_{r,j}, \tau'_{r,j}),
\]
where $\tau'_{r,j}=\tau_{r,j}$ for $r\in[\ell_j,k)$ and $\tau'_{k,j}=R_1$.
Let us estimate $N_{\alpha, \lambda_j,\gamma}(n_{r,j}, R)$ for $r\geq \ell_j$ and $0\leq R\leq\tau_{r,j}$.
The quantity $N_{\alpha, \lambda_j,\gamma}(n_{r,j}, R)$ is equal to the number of solutions of the system 
of inequalities
\begin{align*}
\left\{\frac{x}{\tau_{r,j}}\right\}< &\ \frac{R}{\tau_{r,j}},
\\
\{\alpha \lambda_j^{n_{r,j}+x}\} <&\ \gamma,
\end{align*}
 for  $x=0, \ldots, \tau_{r,j}-1$.
We apply Lemma~\ref{lemma:koksma} with $s=2$, $P=\tau_{r,j}$ and $n=A_{r,j}$ and obtain
\begin{align*}
&\left|N_{\alpha, \lambda_j,\gamma} (n_{r,j}, R)-\gamma \frac{R}{\tau_{r,j}}\tau_{r,j}\right|
\leq 
\\
&\qquad 30^2\left( \frac{\tau_{r,j}}{A_{r,j}}+ 
\sum_{m_1,m_2=-A_{r,j}}^{A_{r,j}} \!\!\!\!\!\!\!\!^{\Large '}
\ \ \ \ \ \ \ 
\frac{1}{\overline{m_1}\overline{m_2}}
\left|
\sum_{x=0}^{\tau_{r,j}-1} 
e\Big(2\pi i 
\Big(m_1  \alpha \lambda_j^{n_{r,j}+x} + \frac{m_2 x}{\tau_{r,j}}\Big)\Big)
\right|
\right).
\end{align*}
Using the definition of $\alpha_r$, for any $r\geq \ell_1$,
\[
\alpha=\alpha_r+\frac{a_r}{2^{n_r} q_r}+ \frac{\theta_r}{2^{n_{r+1}}},
\]
where $ 0\leq\theta_r\leq 2$, because 
\[
    \frac{\theta_r}{2^{n_{r+1}}} = 
    \sum_{k=r+1}^{\infty} \frac{a_k}{2^{n_k} q_k} < 
    \sum_{k=r+1}^{\infty} \frac{1}{2^{n_k}} = 
    \frac{1}{2^{n_{r+1}}} \sum_{k=r+1}^{\infty} \frac{1}{2^{n_k - n_{r+1}}} \leq
    \frac{2}{2^{n_{r+1}}}.
\]
Now, using the definition of $D_{r,j}(a_r)
=\sideset{}{'}\sum_{m_1,m_2=-A_{r,j}}^{A_{r,j}} \frac{|S_{r,j}(m_1,m_2,a_r)|}{\overline{m_1}\ \overline{m_2}}$ 
we obtain,
\begin{align*}
&\left|N_{\alpha, \lambda_j,\gamma} (n_{r,j}, R)-\gamma R \right|
\leq 
\\
&\qquad 30^2\left( \frac{\tau_{r,j}}{A_{r,j}}+ 
D_{r,j}(a_r) + \sum_{m_1,m_2=-A_{r,j}}^{A_{r,j}} \!\!\!\!\!\!\!\!^{\Large '}
\ \ \ \ \ \ \ 
\frac{1}{\overline{m_1}\overline{m_2}} \left| U(m_1,m_2,a_r)\right|
\right)
\end{align*}
where 
\[
\left|U(m_1,m_2, a_r)\right| = \left| S_{r,j}(m_1,m_2,a_r) - \sum_{x=0}^{\tau_{r,j}-1} e\Big(2\pi i \Big(  m_1  \alpha \lambda_j^{n_{r,j}+x} + \frac{m_2 x}{\tau_{r,j}} \Big)\Big)\right|.
\]
By the definition of $S_{r,j}(m_1,m_2,a_r)$, 
 the condition $0\leq \theta_r\leq 2$,
and the fact that for every pair of reals $\xi_1$ and $\xi_2$,
\[
| e(2\pi i \xi_1) -e(2 \pi i \xi_2)|=
2|\sin(\pi(\xi_1-\xi_2)))|
\leq 2\pi |\xi_1-\xi_2|,
\]
we find that 
\begin{align*}
\left|U(m_1,m_2, a_r) \right|
\leq& \ 2\pi \sum_{x=0}^{\tau_{r,j}-1}|m_1| \lambda_j^{n_{r,j}+x} \frac{\theta_r}{2^{n_{r+1}}}
\\
\leq &\  4\pi |m_1| \lambda_j^{n_{r+1},j}  \frac{1}{(\lambda_j-1) 2^{n_{r+1}}}
\\
\leq &\ \frac{4\pi |m_1|}{\lambda_j-1}.
\end{align*}
Then, using that  $ A_{r,j}\leq \sqrt{\tau_{r,j}}$, 
 the upper bound for $D_{r,j}(a_r) $ given  in Lemma~\ref{lemma2} for $r\geq \ell_j$,
and the inequality 
$\sum{^{\Large '}}_{m_1,m_2=-A_{r,j}}^{A_{r,j}}
\frac{1}{\overline{m_1}\overline{m_2}}\leq (3+\ln \tau_{r,j})^2$,
we obtain,
\begin{align*}
&|N_{\alpha,\lambda_j,\gamma}(n_{r,j}, R)-\gamma R| 
\\
&\qquad \leq 30^2\left( 2\sqrt{\tau_{r,j}}+ 2 \Big(\frac{\lambda_j}{\lambda_j-1}\Big)^{3/2}
\sqrt{\tau_{r,j}} (3+\ln \tau_{r,j})^2 \omega(r)+ \frac{4\pi}{\lambda_j-1}\sqrt{\tau_{r,j}}(3+\ln \tau_{r,j})^2\right)
\\
& \qquad \leq 30^2 15 \Big(\frac{\lambda_j}{\lambda_j-1}\Big)^{3/2} \sqrt{\tau_{r,j}}(3+\ln \tau_{r,j})^2 \omega(r).
\end{align*}
For  $k\geq \ell_j$, 
\[
N_{\alpha, \lambda_j,\gamma}(P)= N_{\alpha, \lambda_j,\gamma}(n_{\ell_j,j})+\sum_{r=\ell_j}^k N_{\alpha, \lambda_j,\gamma}(n_{r,j}, \tau'_{r,j}),
\]
and  
\[
P=n_{k,j}+R_1, \text{ where }0\leq R_1< \tau_{k,j}.
\]
So, we have
\[
| N_{\alpha, \lambda_j,\gamma}(P)-\gamma P|\leq |N_{\alpha, \lambda_j,\gamma}(n_{\ell_j,j})-\gamma n_{\ell_j,j}| + 
\sum_{r=\ell_j}^k 30^2 15 \Big(\frac{\lambda_j}{\lambda_j-1}\Big)^{3/2} \sqrt{\tau_{r,j}}(3+\ln \tau_{r,j})^2 \omega(r).
\]
and
\[
P\geq \tau_{k-1,j}\geq \frac{1}{4} \tau_{k,j}.
\]
It follows from  Lemma~\ref{lemma:easy} that 
\[
\sum_{r=\ell_j}^k \sqrt{\tau_{r,j}}
\leq 
\sum_{r=\ell_j}^k \sqrt{2^{r+1} \log_{\lambda_j} 2}
\leq
3 \sqrt{2^{k+2}\log_{\lambda_j} 2}
\leq 10\sqrt{\tau_{k,j}}.
\]
Let us show that, for $k\geq \ell_j$,
\[
\omega(P)\geq \omega(k).
\]
 Since $\omega(r)$ is a non-decreasing sequence, it is sufficient to show that, for  $k\geq \ell_j$,
\[
P\geq k.
\]
In fact, using the definitions of  $\ell_j$ and $n_{k,j}$, and $P=n_{k,j}+R_1$,
\[
k\geq \ell_j\geq 5,\quad 2^{\frac{k+1}{2}}\geq k+1 \mbox{ for }k\geq 5,
\]
and
\begin{align*}
P-k &\geq  n_{k,j}-k
\\&\geq   (2^k-2)\log_{\lambda_j} 2- k-1
\\&\geq  (\log_{\lambda_j}2 )(2^{k-1} -(k+1) \log_2 \lambda_j)
\\& \geq (\log_{\lambda_j}2 )(2^{k-1} -(k+1)  2^{\frac{\ell_j-3}{2}})
\\&\geq 2^{\frac{k-3}{2}}(\log_{\lambda_j }2) (2^{\frac{k+1}{2}} -k-1)
\\&\geq 0.
\end{align*}
Using the inequalities above and the obvious inequality
$|N_{\alpha,\lambda_j\gamma}(n_{\ell_j,j}) -\gamma n_{\ell_j,j}|\leq n_{\ell_j,j}$,
we have
\[
    |N_{\alpha,\lambda_j,\gamma} (P) - \gamma P| \leq 
n_{\ell_j,j} + 4 \ 10^5 \Big(\frac{\lambda_j}{\lambda_j - 1} \Big)^{3/2} \sqrt{P} (5 + \ln P)^2 \omega(P).
\]
The above inequality also holds for $k \leq \ell_j - 1$, since
\[
|N_{\alpha,\lambda_j\gamma} (P)-\gamma P|\leq P < n_{k+1,j} \leq n_{\ell_j,j}.
\]
Recalling the definition of $n_{\ell_j,j}$ we finally obtain
\[
|N_{\alpha, \lambda_j,\gamma}(P)-\gamma P|\leq 
2^{\ell_j} \log_{\lambda_j}2 + 4 \ 10^5 \Big(\frac{\lambda_j}{\lambda_j-1}\Big)^{3/2}\sqrt{P}(5+\ln P)^2\omega(P).
\]
Thus, 
the discrepancy of  the sequence $\{ \alpha \lambda_j^x\}_{x\geq 0}$, for any given positive integer  $P$, 
\[
D(P, \{\alpha \lambda_j^x\}_{x\geq 0} ) = 
\sup_{\gamma\in(0,1]} \left| 
\frac{N_{\alpha, \lambda_j, \gamma}(P)}{P}-\gamma\right|
\text{ is in }  O\Big(\frac{(\log P)^2}{\sqrt{P}} \omega(P)\Big).
\]
\end{proof}

\begin{corollary}[\cite{levin}]
Let $\lambda_j=j+1$, $t_j=2^j$ for $j=1,2,\ldots$, so $\ell_j\leq 2^{j+1} +1$ and  $\omega(P)\leq 2(5+\ln P)$.
Then, the constructed number $\alpha$ is absolutely normal
in Borel's sense, and for any integer $j\geq 2$, 
the discrepancy of $\{\alpha j^x\}$, for $x=0, \ldots, P-1$ is
\[
D(P, \{\alpha j^x\}_{x\geq 0} ) 
\leq \frac{2^{2j+1}}{P} \log_j 2+ 3\ 10^6 \frac{(5+\ln P)^3}{\sqrt{P}},  \]
which is  in $O\Big(\frac{(\log P)^3}{\sqrt{P}}\Big)$.
\end{corollary}

Levin asserts that a similar method can be used for constructing a number~$\alpha$ such that,
given any integer $j$, the discrepancy of the sequence $\{\alpha \lambda_j^x\}_{x=0}^{P-1}$, 
is $O\Big(\frac{(\log P)^{3/2}}{\sqrt{P}}\omega(P)\Big)$, where the constant in the order symbol~$O$ 
depends on $\lambda_j$, and he gives as reference Section~2 of~\cite{levin1975}.

\section{About Levin's construction and its possible variants}

\subsection{Possible variants on the construction}

Here we consider other possible  values for $n_r$ and $q_r$ to run Levin's construction.
This is interesting because smaller values of $n_r$ imply  a faster computation at step $r$,  due to the fact that  $a_r$  is searched in a smaller range.
However, smaller values of $n_r$ imply a larger discrepancy of the sequence $\{\lambda_j^x \alpha\}_{x\geq 0}$.

The next Lemma~\ref{lemma:condqr} gives a sufficient condition for  $n_r$ and $q_r$ to ensure that the construction works.
Then, Lemma~\ref{lemma:condsqrt}   gives a sufficient condition on $n_r$ to ensure that the construction yields an 
absolutely normal number: the value $n_r$ must  be polynomial in $r$, with degree greater than~$1$.

\begin{lemma}\label{lemma:condqr}
    If $\lambda_j \geq 2$ and the sequences $n_1, n_2, \ldots$ and $q_1, q_2, \ldots$ satisfy for every positive $r$, 
\[
    2^{n_{r+1} - n_r + 1 + \frac{1}{2} \log (n_{r+1} - n_r + 1)} \leq q_r
\]
then the statement of Lemma~\ref{lemma:levin1} holds.
\end{lemma}

\begin{proof}
    In Lemma~\ref{lemma:levin1}, every step of the proof is valid disregarding 
    the values chosen for $n_1, n_2, \ldots$ and $q_1, q_2, \ldots$ except for the statement
    \[
        |m_1|(\lambda_j^{n_{r,j} + x} - \lambda_j^{n_{r,j} + y}) \leq \frac{1}{2} 2^{n_r} q_r.
    \]
    We show that the condition given by this lemma is sufficient to make the above inequality true.
    Let us recall that 
    $ n_{r,j} = \lfloor n_r \log_{\lambda_j} 2 \rfloor $,
    $\tau_{r,j} = n_{r+1,j} - n_{r,j}$, $ 0 \leq x,y < \tau_{r,j} $ and
    $ |m_1| \leq A_{r,j} = \lfloor \sqrt{\tau_{r,j}} \rfloor $.
    Then,   
    \begin{eqnarray*}
        q_r & \geq & 2^{n_{r+1} - n_r + 1 + \frac{1}{2}\log_2(n_{r+1} - n_r + 1)} \\
            & = & \sqrt{n_{r+1} - n_r + 1} \ 2^{n_{r+1} - n_r + 1} \\
            & \geq & \sqrt{(n_{r+1} \log_{\lambda_j} 2 - n_r \log_{\lambda_j} 2) + 1} \ 2^{n_{r+1} - n_r + 1} \\
            & \geq & \sqrt{n_{r+1,j} - n_{r,j}} \ 2^{n_{r+1} - n_r + 1} \\
            & = & \sqrt{\tau_{r,j}} \ 2^{n_{r+1} - n_r + 1} \\
            & \geq & |m_1| 2^{n_{r+1} - n_r + 1} \\
            & = & 2 |m_1| 2^{n_{r+1}} {2^{-n_r}} \\
            & > & 2 |m_1| \lambda_j^{n_{r+1,j}} \lambda_j^{-(n_{r,j}+1)} \\
            & = & 2 |m_1| \lambda_j^{n_{r+1,j} - n_{r,j} - 1} \\
            & = & 2 |m_1| \lambda_j^{\tau_{r,j}-1} \\
            & > & 2 |m_1| (\lambda_j^{\tau_{r,j}-1} -1) \\
            & \geq & 2 |m_1| \frac{\lambda_j^{n_{r,j}}}{2^{n_r}} (\lambda_j^{\tau_{r,j}-1} -1) \\
            & \geq & 2 |m_1| \frac{\lambda_j^{n_{r,j}}}{2^{n_r}} (\lambda_j^x - \lambda_j^y) \\
            & = & \frac{2}{2^{n_{r,j}}} |m_1| (\lambda_j^{n_{r,j} + x} - \lambda_j^{n_{r,j} + y}).
    \end{eqnarray*}
\end{proof}

In what follows we use  customary asymptotic notation to describe the growth rate of the functions.
We write,
\\\begin{tabular}{lll}
$f(n) $ is in $\Theta(g(n))$ & if $\exists k_1>0\ \exists k_2>0\ \exists n_0\ \forall n>n_0$, & $ k_1 g(n) \leq f(n) \leq k_2 g(n)$,
\\
$f(n)$ is in $o(g(n))$ &if $	\forall k>0 \ \exists n_0 \ \forall n>n_0$, & $ |f(n)| \leq k |g(n)|$.
\end{tabular}

\begin{lemma}\label{lemma:condsqrt}
    Let $j$ and $P$ be positive integers and let $k$ be such that $n_{k,j} \leq P < n_{k+1,j}$.
    If $\sum_{r=1}^k \sqrt{n_{r+1,j} - n_{r,j}}$ is in 
$o\Big(\frac{P}{(\log P)^2 \omega(P)}\Big)$ then Levin's construction yields an absolutely normal number.
\end{lemma}
\begin{proof}
    See proof of Theorem~\ref{thm:levin} for the upper bound of $|N_{\alpha,\lambda_j,\gamma}(P) - \gamma P|$.
\end{proof}

The next proposition shows that if  $n_r$  dominates  any  linear function on $r$, 
and $q_r$ is increasing in $r$ according to a condition in the  the growth of $n_r$,
then  Levin's construction yields an absolutely normal number,

\begin{proposition}\label{prop:nr}
Let $(\lambda_j)_{j\geq 1}$ be a sequence or real numbers greater than $1$ and let 
$(t_j)_{j\geq 1}$ be a sequence of reals such that the function  $\omega(P)$ has  sub-polynomial growth.
If  $n_r$ is  any  polynomial on $r$ with degree greater than~$1$, 
and $q_r$ is such that  
\[
    n_{r+1} - n_r + 1 + \frac{1}{2} \log (n_{r+1} - n_r + 1) \leq \log q_r,
\]
then Levin's construction yields an absolutely normal number. 
However, if $n_r$ is linear in $r$, Levin's arguments do not prove that the discrepancy goes to~$0$. 
\end{proposition}

\begin{proof}
Suppose   $n_r$ is polynomial on $r$.
Then, there is some $h$ such that  $n_r$ in $\Theta(r^h)$.
By definition of $n_{r,j}$, we have
$n_{r,j} = \lfloor n_r \log_{\lambda_j} 2 \rfloor $ is in $\Theta (r^h)$.
Hence,   $n_{r+1,j} - n_{r,j}$ is in $\Theta(r^{h-1})$;    
therefore, $\sqrt{n_{r+1,j} - n_{r,j}}$ in $\Theta(r^{\frac{h-1}{2}})$.
Furthermore, if $P$ and $k$ are such that $n_{k,j} \leq P < n_{k+1,j}$, then $k$ is in $\Theta(\sqrt[h]{P})$.
Thus,
\[
\sum_{r=1}^k \sqrt{n_{r+1,j} - n_{r,j}} \text{ is in } \Theta \Big(\big(\sqrt[h]{P}\big)^{\frac{h+1}{2}}\Big) = \Theta\big(P^{\frac{h+1}{2h}}\big).
\]
If $n_r$ were   a linear function on $r$,
\[
\sum_{r=1}^k \sqrt{n_{r+1,j} - n_{r,j}} \text{ is in }\Theta(P),
\] 
hence $\sum_{r=1}^k \sqrt{n_{r+1,j} - n_{r,j}}$ would not be in the required class $o\Big(\frac{P}{(\log P)^2 \omega(P)}\Big)$.
We conclude that, to obtain a normal number with Levin's construction,   $n_r$ can not be linear in $r$.
Instead, $n_r$ can be any  polynomial on $r$ with degree greater than~$1$ 
provided that  $\omega(P)$ is chosen to have  sub-polynomial growth.
\end{proof}

In Levin's construction smaller values  of  $n_r$  imply a larger upper bound on discrepancy of the sequence $\{\alpha \lambda^x\}$.
The following table shows the bound for the discrepancy of the sequence $\{ \lambda_j^x \alpha\}_{x= 0}^{P}$, obtained using Levin's proof  for different choices of $n_r$.
In each case the constant behind the $O$ symbol depends on $\lambda_j$.
\begin{center}
\begin{tabular}{c | l}
    $n_r$ & Discrepancy bound given by Levin's proof\\
    \hline
\\ 
    $r$ & $O( \log(P)^2 \omega(P) )$   ---it does not go to $0$ when $P$ goes to $\infty$---
\\\\
    $r^h$ & $O\Big(\frac{\log(P)^2 \omega(P)}{P^{\frac{h-1}{2h}}}\Big)$ 
\\\\
    $2^r - 2$ & $O\Big(\frac{\log(P)^2 \omega(P)}{\sqrt{P}}\Big)$ 
\end{tabular}
\end{center}

In all these cases, the upper bound for discrepancy  contains  $\omega(P)$, as in Levin's formulation
and the  constant hidden in the $O$ symbol depends on the base $\lambda_j$.
Although Levin states that for any nondecreasing function $\omega(P)$ his construction produces
an absolutely normal real number,  the growth of $\omega(P)$ cannot be arbitrary.
For example, when $n_r$ is $2^r-2$,  $\omega(P) = \sqrt{P}$   does not give a discrepancy bound going to $0$.

\subsection{Necessary conditions on the construction}

Levin's construction is not  conceived as the concatenation of the binary expansions of the $a_r$ for $r = 1,2,\ldots$. 
This means that the expansion in base $2$ of  $\alpha_{r+1}$ is {\em not} obtained as a concatenation of 
the expansion of $\alpha_r$ with the base-$2$ representation of $a_{r}$.
Recall the definition of $\alpha_{r+1}$:   $\alpha_{\ell_1}$ is equal to a starting real number $a$ (argument for the construction)
and for every $r\geq \ell_1$,
\[
\alpha_{r+1}=\alpha_{\ell_1}+\sum_{m={\ell_1}}^r    \frac{a_m}{2^{n_{m}} {q_m}},
\]
where $a_m$ is an integer in $[0, q_m)$ satisfying the conditions of Lemma \ref{lemma2},
\[
n_m=2^m-2  \text{ and } q_m= 2^{2^m+m+1}.
\]
Since 
$ \log q_r = 2^r+r+1 > n_{r+1} - n_r = 2^r$
we have
\[ \alpha_{r+1} - \floor{2^{n_{r+1}}\alpha_{r+1}}2^{-n_{r+1}} > 0.
\]

The next Lemma~\ref{lemma:overlap} shows that if $q_r$ is unbounded,  
then it is necessary for Levin's proof that $ q_r > 2^{n_{r+1} - n_r}$.
This condition is implied by the sufficient condition on $n_r $ and $q_r$ we identified in Lemma~\ref{lemma:condqr}.

Then, Lemma \ref{lemma:qrbounded} proves that if $q_r$ id bounded then Levin's construction does not yield an absolutely normal number.

\begin{lemma}\label{lemma:overlap}
    If  $q_r$ is unbounded then it is necessary that $q_r > 2^{n_{r+1} - n_r}$.
\end{lemma}

\begin{proof}
For Lemma~\ref{lemma:levin1} to hold, we need that $|m_1| (\lambda_j^{n_{r,j} + x} - \lambda_j^{n_{r,j} + y}) \leq \frac{1}{2} 2^{n_r} q_r$.
In particular, when $m_1 = A_{r,j}$, $x = \tau_{r,j}-1$, $y = 0$ and $\lambda_j = 2$, we need
that  the following inequality holds:
\[
    A_{r,j} (2^{n_{r+1}-1} - 2^{n_r}) \leq \frac{1}{2} 2^{n_r} q_r. 
\]
Equivalently,
\[
    A_{r,j} (2^{n_{r+1} - n_r} - 2) \leq q_r.
\]
Now  suppose that,
$q_r$  is unbounded, non-decreasing in $r$ but, 
 contrary to the statement of the Lemma,
$ q_r \leq 2^{n_{r+1} - n_r}$.
So,  the above condition becomes
\[
 A_{r,j}  \leq \frac{q_r}{(q_r - 2)}.
\]
Since $q_r$ is unbounded, there is  $r_0$ such that  for every $r\geq r_0$, $q_r \geq 1000$,
and  each of the following inequalities should hold.
\begin{align*}
    A_{r,j} \leq \frac{q_r}{q_r - 2} \leq \frac{1000}{998} <&2
\\
    A_{r,j}=\lfloor \sqrt{n_{r+1} - n_r} \rfloor <& 2
\\
    \sqrt{n_{r+1} - n_r} <& 3
\\
    n_{r+1} - n_r <& 9
\\
    2^{n_{r+1} - n_r} <& 512.
\end{align*}
Then, using the assumption $ q_r \leq 2^{n_{r+1} - n_r}$,
we conclude $q_r<512$, 
contradicting that  $q_r \geq 1000$.
\end{proof}

The following lemma shows that if $q_r$ is bounded by a constant, then Levin's construction 
yields a number $\alpha$ which might not be absolutely normal.

\begin{lemma}\label{lemma:qrbounded}
If $q_r$ is bounded by a constant and $\log q_r \leq n_{r+1} - n_r$ then Levin's construction 
does not ensure absolute normality.
\end{lemma}
\begin{proof}
For ease of presentation assume  the argument  $a$ in Levin's construction  is a 
non-negative rational  number of the form $r+ p/2^{\ell_1}$ for some 
non-negative integers $r$ and $p$ with $p$ less than $\ell_1$.
So, the  expansion of $a$ in base $2$ has at most $2^{\ell_1}$  significant digits.
The case where $a$  is not of this form  can be proved similarly.

Suppose that $q_r$ is bounded, then $\log_2 q_r$ will be bounded too. That is, there is a constant $C$
such that for all $r$, $\log_2 q_r \leq C$.
Since $a_r$  is in $[0, q_r)$, at step $r$, the choice of $a_r$ requires at most $C$ binary digits. 
Suppose also that  $\log_2 q_r \leq n_{r+1} - n_r$, which implies that the fractions $\frac{a_r}{2^{n_r} q_r}$ 
have binary expansions that  not overlap. 
Therefore, the first $n_r$ bits of $\alpha$ will be correctly computed on the $r$-th step of the construction,
that is the  first $n_r$ bits of $\alpha$ will coincide with those of  $\alpha_r$.

Let $0.b_0 b_1 b_2 \ldots$ be the binary expansion of $\alpha$ and let $\text{count}(\alpha, n,b)$ be the number of
bits equal to $b$ within $b_0, b_1, \ldots, b_{n-1}$. Assuming $\alpha$ is absolutely normal, it must be simply normal in base $2$.
Therefore $\lim_{n \to \infty} \frac{\text{count}(\alpha,n,b)}{n} = \frac{1}{2}$
 for $b$ in $\{0,1\}$.
Using the definition of limit, for all positive $\epsilon$, for every sufficiently large $r$,
\[
    \frac{1}{2}-\epsilon <  \frac{\text{count}(\alpha,n_r,1)}{n_r} \leq \frac{C(r-1)}{n_r}.
\]
So,
\[
    n_r < \frac{Cr - C}{\frac{1}{2} - \epsilon}.
\]
Since this holds all positive $\epsilon$, we conclude
\[ n_r \text{ in } O(r). 
\]
On the other hand, we can safely assume that $\log_2 q_r \geq 1$ because at least $1$ bit should be computed 
on each step of the algorithm. Given that $n_{r+1}-n_r \geq \log_2 q_r \geq 1$ we obtain 
that $n_r$ must be in $\Theta(r)$.
As we stated on Proposition~\ref{prop:nr}, a linear growth of $n_r$ does not ensure that discrepancy 
goes to $0$. So we cannot ensure absolute normality of the generated number~$\alpha$.
\end{proof}

More importantly, this necessary  condition on $n_r$ and $q_r$
determines that Levin's construction of the number $\alpha$ is not doable as a concatenation of the $a_r$, for $r=1,2, \ldots$.

\begin{proposition}\label{prop:concatenation}
If $q_r$ and $n_r$ are such that $\log q_r > n_{r+1} - n_r  $
then Levin's  construction of $\alpha$ is not doable as the concatenation of the $a_r$, for $r=1,2, 3\ldots$.
\end{proposition}
\begin{proof}
To run the construction as a concatenation of the $a_r$, for $r = 1,2,3,\ldots$, we need that 
\[
    \sum_{m=0}^{r-1} \log q_m \leq n_r.
\]
But
\[
    \sum_{m=0}^{r-1} \log q_m > \sum_{m=0}^{r-1} n_{m+1} - n_m = n_r - n_0 = n_r.
\]
\end{proof}

\section{Levin's normal numbers are computable}

The theory of computability defines a computable function from non-negative integers to non-negative
integers as one which can be effectively calculated by some algorithm.  The definition extends to
functions from one countable set to another, by fixing enumerations of those sets.  A real number $x$ is
computable if there is a base and a computable function that gives the digit at each position of the
expansion of $x$ in that base.  Equivalently, a real number is computable if there is a computable
sequence of rational numbers $(r_n)_{n\geq 0}$ such that $|x -r_n| < 2^{-n}$ for each~$n\ge 0$.

\begin{theorem}[Turing   {\cite[Theorem 5.1.2]{downeyhirschfeldt}}]\label{thm:computable}
    The following are equivalent:
    \begin{enumerate}
        \item The real $x$ is computable.
        \item There is a computable sequence of rationals $(r_n)_{n\geq 0}$ that tends to $x$ such that 
            $|x - r_n| < 2^{-n}$ for all $n$.
        \item There is a computable sequence of rationals 
            $(r_n)_{n\geq 0}$ that converges to $x$ and a computable function 
            $f:\N \to \N$ such that $|x - r_{f(n)}| < 2^{-n}$ for all $n$.
    \end{enumerate}
\end{theorem}

\begin{theorem}\label{thm:computability}
Let $(\lambda_j)_{j\geq 1}$ be computable sequence of integers greater than $2$ 
 let $(t_j)_{j\geq 1}$ be a computable sequence of integers monotonically increasing at any speed
and let the starting value $a$ be  a rational number,
Then, the number $\alpha$ defined by Levin, proved to be absolutely normal in Theorem~\ref{thm:levin}, is computable.
\end{theorem}
\begin{proof}
The number $\alpha$ is the limit of $\alpha_r$ for $r$ going to infinity,  where
$a_{\ell_1}=a$ with 
\\
$\ell_1=  \max(t_1,  2 \lceil |\log_2\log_2 \lambda_1| \rceil + 5)$, 
and for $r\geq 1$,
\[
\alpha_{r+1}=\alpha_r +   \frac{a_r}{2^{n_{r}} {q_r}},
\]
where
$a_r$ is an integer in $[0, q_r)$ satisfying the inequalities of Lemma~\ref{lemma2},
$n_r=2^r-2$ and 
$q_r= 2^{2^r+r+1}$.
Lemma~\ref{lemma2}  proves  that such $a_r$ exists. 
Since   $D_{r,j}(c)$ is a computable function  it is possible to 
find $a_r$ by an exhaustive search among all integers in $[0, q_r)$ and all bases $\lambda_j$ for $j=1,2,\ldots, \omega(r)$,
where         $ \omega(r) = 1$ if  $r $ in $[1,\ell_2)) $, otherwise   $\omega(r)$ is  the  unique index  $k $ such that  $r $ in $ [\ell_k, \ell_{k+1})$,
with 
$\ell_k = \max(t_k, \max_{1 \leq v \leq k} 2 \lceil |\log_2\log_2 \lambda_v| \rceil + 5)$.
At each step $r$, we can compute bitwise approximations  of $D_{r,j}$ from above,  for each of the 
possible candidate values of $a_r$ until we find one that satisfies the requires inequality for all~$j$  between 
$1$ and $\omega(r)$.
Thus, the  sequence of rationals $\alpha_1, \alpha_2, \ldots$ is computable and converges 
to  an absolutely normal  number~$\alpha$.
From the proof of Theorem \ref{thm:levin} we know that, for each~$r$, 
\[
\left|\alpha - \alpha_r\right| < \frac{2}{2^{n_r}}. 
\]
Since $\alpha$ is an absolutely normal number, and therefore an irrational number, 
by Theorem \ref{thm:computable} we conclude that $\alpha$ is computable.
\end{proof}

\section{The computational complexity of Levin's construction}

Theorem \ref{thm:computability}  proves that under some assumptions of the sequences $(\lambda_j)_{j\geq 1}$ and $(t_j)_{j\geq 1}$, and the starting value $a$, 
 Levin's construction  is indeed an algorithm to compute the number~$\alpha$.
The algorithm is recursive. 

The standard computational model is the Turing machine model, which works  just with   
finite representations, so it only  deals with  numbers that are the limit of a computable sequence of finite approximations. 
In this model, at step $r$, the number of elementary  operations needed to find out the  
number~$a_r$ can not be easily determined.
This is because to find out $a_r$ the algorithm must  compute sums of exponential sums. 
The terms in these sums are transcendental numbers, which can only be computed as limits of  finite approximations. 
It is impossible to determine how many approximations to each term of the exponential sums 
must be computed to find out that a candidate $a_r$ is conclusive. 

So, instead of counting the number of elementary operations needed to compute the number $a_r$ at step~$r$,
here we give the number of mathematical  operations needed in  an idealized computational 
model over the real numbers, based on machines with infinite-precision real numbers.  
A canonical model for this form of computation over the reals  is Blum-Shub-Smale machine~\cite{blumshubsmale}, abbreviated  BSS machine.
This is a machine with registers that can store arbitrary real numbers and  can compute rational functions over reals at unit cost.
Since  elementary trascendental functions, as exponential function or trigonometric functions, 
are not computable by a BSS machine we need to consider the extended  BSS machine 
which includes exponential and trigonometric functions as primitive operations. 
For our purpose, the extended BSS model is identical to considering Boolean 
arithmetic circuits augmented  with trigonometric functions.

Of course, for any given real valued function, its  complexity in the BSS model gives just 
a lower bound  of its complexity in the classical Turing machine model, where 
the  cost for  arithmetic (and trigonometric) operations over the real numbers  is not constant.

\begin{theorem}\label{thm:complexity}
    Let $(\lambda_j)_{j\geq 1}$ be a computable sequence of reals greater than~$1$ and let $(t_j)_{j\geq 1}$ 
 be a computable sequence of integers.
Levin's algorithm requires 
\[
O\Big(2^{2^r + 3r + 1} \sum_{j=1}^{\omega(r)} (\log_{\lambda_j} 2)^2\Big)
\] 
mathematical operations to compute $\alpha_r$, for each $r$. 
\end{theorem}

\begin{proof}
    Assume a BSS machine  which includes exponential and trigonometric functions as primitive operations. 
    The expression $S_{r,j}(m_1,m_2,c)$ is the sum of $\tau_{r,j}$ terms, 
     each of them can be computed in constant time in our machine. 
    Hence the time needed to compute each value of $S_{r,j}$ is in $O(\tau_{r,j})$.

    To obtain a value of $D_{r,j}$ we must calculate $O({A_{r,j}}^2) = O(\tau_{r,j})$ values of $S_{r,j}$.
    Therefore, the computation of $D_{r,j}$ is in $O({\tau_{r,j}}^2) = O((2^r \log_{\lambda_j} 2)^2)$.

    Finding the value of $a_r$ requires to compute $D_{r,j}(c)$ for  each $j$ between $1$ and $\omega(r)$
    until we find a value of $c$ in $[0,q_r)$ which satisfies the inequalities of Lemma~\ref{lemma2}.
   In the worst case, it will be necessary to try all possible values for $c$. 
   In this worst case, the required time 
    is in  
\medskip
\\   
$O\Big(\sum_{c=0}^{q_r-1} \sum_{j=1}^{\omega(r)} (2^r \log_{\lambda_j}2)^2\Big) = 
  O\Big(q_r \sum_{j=1}^{\omega(r)}(2^r \log_{\lambda_j}2)^2\Big) =
  O\Big(2^{2^r + 3r + 1} \sum_{j=1}^{\omega(r)}(\log_{\lambda_j}2)^2\Big).$
\medskip
\\
Let $C_k$ be   the time required to compute $a_k$,
   \[
    C_k = 2^{2^k + 3k + 1} \sum_{j=1}^{\omega(k)} (\log_{\lambda_j} 2)^2
   \]
   Then, the  time to compute $\alpha_r$ is 
   $\sum_{k=1}^r C_k$. 
   We now show that the time needed  to compute $\alpha_r$ is 
   essentially the time spent on the search of $a_r$. 
   We need to show that   $\sum_{k=1}^{r} C_k $ is in $O(C_r)$, because 
   \begin{align*}
           \sum_{k=1}^{r-1} C_k \ \leq \  (r-1) C_{r-1} 
      \ = \ (r-1) 2^{2^{r-1} + 3(r-1) + 1} \sum_{j=1}^{\omega(r-1)} (\log_{\lambda_j} 2)^2
   \end{align*}
   which is in 
   $ O\Big(2^{2^r + 3r + 1} \sum_{j=1}^{\omega(r)} (\log_{\lambda_j} 2)^2\Big)$.
\end{proof}

Notice that Theorem \ref{thm:complexity} estimates the complexity of  obtaining a rational approximation $\alpha_r$ 
with an error bounded by $2^{-(n_{r+1}-1)}$. 
Since  $\alpha_r$ is just an approximation to $\alpha$, 
it is not determined how many  bits in the expansion of $\alpha_r$ are conclusive so as to conform the expansion of $\alpha$.
One would like that the   first $n_{r+1}-1$ bits of $\alpha_r$ determine those of $\alpha$.
As we showed in Proposition \ref{prop:concatenation} Levin's  construction is not doable as the concatenation of the values $a_r$.
An overlapping of the fractions $\frac{a_r}{2^{n_r} q_r}$ may occur, causing carries and changing some of the first bits of $\alpha_r$.

Theorem \ref{thm:complexity} proves that the complexity of computing $\alpha_r$ 
with Levin's original formulation for $n_r$ and $q_r$, is double exponential in $r$.
Since $n_r$ is the  number of bits of $\alpha_r$  that are obtained at step $r$,
and in Levin's original formulation $n_r$ is $2^r-2$,  it is fair to say that  the complexity of Levin's algorithm  is simply exponential in 
the number  of bits computed at step $r$.

We now prove that, in case $n_r$ is quadratic in $r$,
 then Levin's algorithm requires a number of operations that is is simply exponential in 
the square root of number of bits computed at step $r$.

\begin{theorem}
The alternative of Levin's construction with $n_r = r^2$ takes 
\[
O\big(r^3 2^{2r} \sum_{j=1}^{\omega(r)} (\log_{\lambda_j} 2)^2\Big)
\] 
mathematical operations in an extended  BSS machine  to compute $\alpha_r$.
\end{theorem}

\begin{proof}
    First, we need to choose values for $q_r$ that ensure normality. 
    As we showed in Lemma~\ref{lemma:condqr}, 
    $2^{n_{r+1}-n_r+1+\frac{1}{2}\log(n_{r+1}-n_r+1)} \leq q_r$ is a sufficient condition.
    We choose $q_r = 2^{2r + 2 + \lceil \log(2r+2) \rceil}$.
    By Theorem~\ref{thm:complexity}, 
          to find $a_r$,     in the worst case 
   it is necessary  to compute $D_{r,j}(c)$ for each $j$ between $1$ and $\omega(r)$ and  for each $c$ between $0$ and $ q_r-1$
 and each  $D_{r,j}$ requires $O(\tau_{r,j}^2)$ operations.
 Then, the number of operations to find $a_r$   is in  $O(q_r \sum_{j=1}^{\omega(r)} \tau_{r,j}^2)$, because
   \medskip

    $ q_r \text{ is in } O(r 2^{2r})$,

     $\tau_{r,j} \text{ is in } O(r \log_{\lambda_j}2) $, and 

     $ O\Big(q_r \sum_{j=1}^{\omega(r)} \tau_{r,j}^2\Big) = O\Big(r^3 2^{2r} \sum_{j=1}^{\omega(r)} (\log_{\lambda_j} 2)^2\Big).$
\medskip
\\
     The time to compute $\alpha_r$ is essentially the time required to find  $a_r$ because
    \[
      \sum_{k=1}^r k^3 2^{2k} \sum_{j=1}^{\omega(k)} (\log_{\lambda_j} 2)^2 \leq 
       r^3 2^{2r+2}  \sum_{j=1}^{\omega(r)} (\log_{\lambda_j} 2)^2 
      \]
 which is in $      O\Big(r^3 2^{2r}  \sum_{j=1}^{\omega(r)} (\log_{\lambda_j} 2)^2\Big)$.
\end{proof}
\bigskip
\bigskip

\noindent
{\bf Acknowledgements.}
The authors are grateful to  Igor Shparlinski  for suggesting us in 2013 (email communication) to determine the computational complexity 
of Levin's constructions of absolutely normal numbers  and  he explicitely asked  whether  the delivery of digits was in polynomial time.
The two  authors are members of the Laboratoire International Associ\'e INFINIS, CONICET/Universidad de Buenos Aires–CNRS/Universit\'e Paris Diderot.

\bibliographystyle{plain}
\bibliography{levin}
\end{document}